\documentclass[]{interact}

\usepackage{amssymb,amsmath,amsfonts,amsthm,amscd,latexsym,verbatim,graphics,epsfig,indentfirst,xcolor,bbm,faktor,hyperref}
\hypersetup{ hidelinks, }

\let\<\langle
\let\>\rangle
\def\oo{\mathrel { \circ }}
\def\ooo#1{\mathrel { {}_{(#1)}}}

\def\Cur{\mathop {\fam0 Cur}\nolimits}

\def\gc{\mathop {\fam0 gc}\nolimits}
\def\Cend{\mathop {\fam0 Cend}\nolimits}
\def\Cendf{\mathop {\fam0 Cend_{fin}}\nolimits}
\def\End{\mathop {\fam0 End}\nolimits}
\def\id{\mathop {\fam0 id}\nolimits}
\def\ad{\mathop {\fam0 ad}\nolimits}

\def\max{\mathrm{max}}

\def\Pois{\mathrm{Pois}}

\theoremstyle{plain}
\newtheorem{theorem}{Theorem}[section]
\newtheorem{lemma}[theorem]{Lemma}
\newtheorem{corollary}[theorem]{Corollary}

\theoremstyle{definition}
\newtheorem{definition}[theorem]{Definition}
\newtheorem{example}[theorem]{Example}

\theoremstyle{remark}
\newtheorem{remark}{Remark}

\begin{document}

\title{On universal conformal envelopes for quadratic 
conformal algebras}

\author{
\name{R.~A. Kozlov\textsuperscript{a}\textsuperscript{b}\thanks{CONTACT R.~A. Kozlov. Email: KozlovRA.NSU@yandex.ru. 
The work is supported by Mathematical Center in Akademgorodok, grant №075-15-2022-282.}}
\affil{\textsuperscript{a}Sobolev Institute of Mathematics, Novosibirsk, Russian
Federation}
\affil{\textsuperscript{b}Novosibirsk State University, Novosibirsk, Russian
Federation}
}

\maketitle

\begin{abstract}
We prove that every quadratic Lie conformal algebra constructed on a 
special Gel’fand--Dorfman algebra embeds into
the universal enveloping associative conformal 
algebra
with a locality function bound $N = 3$.
\end{abstract}


\section{Introduction}

Conformal Lie algebras were introduced in \cite{KacVA} but the idea was around
for somewhat fifteen years, since quite related essences, vertex algebras, appeared as a
description of the operator product expansion (OPE) of chiral fields in the 2-dimensional
conformal field theory \cite{BPZ}. In particular, conformal Lie algebras originally
described the most intriguing, singular part of the OPE. Later Borcherds 
\cite{Borch} formulated
the axiomatic notion of a vertex algebra, with different motivation though. 

From the more algebraic point of view it turns out that a 
vertex algebra naturally 
carries a conformal Lie algebra
structure, perfectly matching with the case of OPE. 
More closely
this connection was investigated in \cite{Ro1999}, \cite{Ro2000}. 
In light of this, conformal Lie algebras
are indeed a powerful tool for studying vertex algebras and related matters.

In this paper we stress our attention on quadratic conformal Lie algebras, 
a wide subclass, including many interesting representatives such as 
current conformal algebras, 
Virasoro, Neveu--Schwartz, Heisenberg--Virasoro, Virasoro--Schr\"odinger, etc. The main feature of quadratic conformal Lie
algebras
is that they are in one-to-one correspondence with the class of 
so called Gel'fand--Dorfman algebras (GD-algebras).
For specifics see \cite{Xu2000}, \cite{Xu2003} (also \cite{HongWu}, \cite{GD}).

First appeared as an auxiliary tool, later, however, conformal algebras turn into independent area of study. The latter is motivated by the fact that a conformal algebra (over a field $\Bbbk $, $\mathrm{char}\,\Bbbk=0$)
is just an
algebra
in the appropriate (pseudo-tensor) category $\mathcal{M}^*(\Bbbk [\partial])$ of
modules over the polynomial algebra $\Bbbk [\partial] $ 
in one variable \cite{BDK}.
Within the framework of this approach, it is possible to expand the basic concepts, for example, what is an algebra, homomorphism, ideal, representation, module, cohomology.
Therefore, the notion of a conformal algebra is the natural expansion of the notion of a “classical” algebra over a field to the particular pseudotensor category.

While Lie conformal algebras are related to vertex algebras, the other varieties
of conformal algebras somehow illuminate the structure 
of infinite-dimensional
algebras in the corresponding “classical” variety.

In the case of ordinary algebras, there is the well-known way 
of turning an associative algebra into a 
Lie one and, vice versa, of embedding a Lie algebra into its universal  enveloping associative algebra. 
Though the first construction could be easily generalized on the conformal level, the latter, in contrary, proves to be way more complicated \cite{Ro2000}, due to some features of the
multiplication. 
This is still an open problem whether every finite (i.e., finitely
generated as a $\Bbbk [\partial]$-module) 
Lie conformal algebra embeds into an
associative conformal algebra with respect to the conformal commutator. 
Even for the
class of quadratic conformal algebras it remains unknown in general
if every such Lie conformal algebra embeds into an appropriate associative one.

A routine way to solve this kind of problem is to construct the universal enveloping. 
In general, such an algebra is defined by generators and relations. 
For a Lie conformal
algebra $C$, there exists a lattice of universal associative enveloping conformal algebras $\mathcal U(C;N)$,
each related to an (associative) locality bound $N$ on the generators \cite{Ro2000}. 
In order to prove (or disprove) the embedding of a Lie conformal algebra into its
universal associative enveloping conformal algebra one needs to know the
normal form of elements in the last one.

A general and powerful method for finding normal forms 
in an algebra defined by
generators and relations is to calculate its standard (or Gr\"obner--Shirshov) basis (GSB) of
defining relations. The idea goes back to Newmann’s Diamond Lemma \cite{New53}, 
see also \cite{Berg1}, \cite{Bok1976}. 
For associative conformal algebras it was initially invented
in \cite{BFK}, later developed in \cite{Kol2020}, \cite{NC} and 
\cite{KK2022}.

Though sometimes this approach is related with technical difficulties, 
due to the rules in a GSB might be highly dependent on the multiplication table of a respective conformal algebra.
This is the case when a quadratic conformal algebra has a rich Novikov structure.
Inspired by \cite{Kol2020IJAC}, \cite{KKP}, in this paper we present an alternative
route to verify that a quadratic conformal algebra with an additional condition embeds into the universal associative envelope related to a
locality function with an upper bound $N = 3$. 

There are several reasons why we consider locality bound $N = 3$, that could be put in
brief as follows: this is the lowest locality that fully reveal all the plethora of
conformal algebras. Indeed, given a Lie algebra $L$, the structure of 
$\mathcal U(\Cur{L}; N)$
for $N = 1$ is pretty much determined by $\Cur{U(L)}$. 
Moreover, even if $L$ is simple and
finite-dimensional, there is a nontrivial central extension of $\Cur{L}$,
namely, the Kac--Moody conformal algebra $K(L)$. Whereas $\mathcal U(\Cur{L}; N)$
has nontrivial central
extension for $N = 2, 3$, namely the universal enveloping of $K(L)$ (see 
\cite{KK2022}). 

Quite similar result was shown in \cite{AK}, the universal enveloping of
the Virasoro conformal algebra has nontrivial second Hochschild cohomology group
(equivalently, nontrivial central extensions) for locality $N = 3$, but not for 
$N = 2$ \cite{Koz2017}.

This paper is organised as follows. In section 2 we give a brief introduction into Novikov and Gel'fand--Dorfman algebras. In section 3.1 the sufficient 
foundation on the theory of Lie and associative conformal algebras is given. 
Some crucial remarks are presented. In section 3.2 we explicitly introduce a conformal
algebra that is the key to the goal of this paper.

In section 4 the main part of this work is presented. 
Section 4.1 is devoted to a
technique that allows us to get through some obstacles (the absence of the PBW property) 
in constructing of the universal enveloping associative conformal 
algebra for a quadratic
conformal algebra with a nontrivial Novikov structure. 
The main statement is formulated.

Section 4.2 is filled with examples that reveal initially hidden feature
of the presented construction: sometimes we are even able to obtain 
an explicit presentation 
of the universal enveloping conformal algebra itself.
Note that for the Virasoro conformal algebra $\mathrm{Vir}$, its 
universal associative envelope for $N=2$ is known as the Weyl conformal algebra \cite{Ro1999}. The universal envelope corresponding to $N=3$ did not have a 
similar nice presentation. However, it turns out that our construction 
being applied to the Virasoro conformal algebra produces 
the universal associative conformal envelope of $\mathrm {Vir}$
corresponding to the locality bound $N=3$.
This is not a general fact: for the abelian Lie conformal algebra of rank one 
our construction produces an envelope which is not universal.

\section{Preliminaries on GD algebras}

\begin{definition}
A vector space $V$ equipped with a bilinear operation $(\cdot\oo\cdot)$: $V\otimes V 
\to V$ is said to be a \textbf{Novikov algebra} if the operation is left-symmetric and
right-commutative:
$$
(a \circ b) \circ c - a \circ (b \circ c) = (b \circ a) \circ c - b \circ (a \circ c),
$$
$$
(a \circ b) \circ c = (a \circ c) \circ b,
$$
for every $a, b, c \in V$.
\end{definition}

\begin{definition}
A vector space $V$ with two bilinear operations $(\cdot\oo\cdot)$, $[\cdot~,~\cdot]$:
$V\otimes V \to V$ is called a \textbf{Gel'fand--Dorfman algebra} (or GD algebra, for short) if it is a Novikov
algebra with respect to $(\cdot\oo\cdot)$, a Lie algebra with respect to
$[\cdot~,~\cdot]$, and there is an additional compatibility condition:
\begin{equation}\label{GD}
[a, b \oo c] - [c, b \oo a] - b \oo [a, c] + [b, a] \oo c  - [b, c] \oo a = 0,    
\end{equation}
for every $a, b, c \in V$.
\end{definition}

Originally introduced in \cite{GD}, both these classes of algebras are used for constructing
Hamiltonian operators in the formal calculus of variations. An abundant source of GD 
algebras are differential Poisson algebras. Recall, that the Poisson algebra is an
associative and commutative algebra along with a Lie bracket $\{\cdot~,~\cdot\}$
subjected to the Leibniz rule:
$$
\{a, b\cdot c\} = b\cdot \{a, c\} + \{a, b\}\cdot c, \quad \forall a, b, c \in P.
$$
Given a Poisson algebra $P$ with a derivation $d$, one can
turn it into a GD
algebra assigning $a\cdot d(b) = a \oo b$ and $\{a, b\} = [a, b]$. It is 
straightforward to check that \eqref{GD} is fulfilled. Denote the resulting GD algebra
as $P^{(d)}$.

\begin{definition}
A GD algebra $V$ is called special if there exists a Poisson algebra $P$
with a derivation $d$, such that there is the injective inclusion $V \subseteq P^{(d)}$.
\end{definition}

Respective $P$ is called a Poisson enveloping for $V$. Those GD algebras that 
are not special are called $\textbf{exceptional}$. 
Examples of exceptional GD algebras may be found in \cite{KBO}, \cite{KKP}.

\begin{example} \label{ex1}
The most natural example of a special GD algebra, that is not originally a Poisson
algebra, is the following one: consider a Novikov algebra $V$ and enrich it with a Lie 
structure by the rule $[a, b] = a \oo b - b \oo a$, for all $a, b \in V$. Denote it as
$V^{(-)}$. Then one
obtain a GD algebra (see \cite{GD}), that is special (\cite{KKP}, a general
construction was presented, see theorems 1, 2 and lemmas 1-3). 

One can construct universal differential enveloping $U_d(V)$ 
of an arbitrary Novikov
algebra $V$ (this enveloping exists due to \cite{BokCZ17}) and then endow it
with a suitable Poisson bracket. Broadly sketching, one should fix
a linear basis $X$ of $V$ along with all formal derivatives 
$d^{\omega}X=\{a^{(n)} \mid a\in X,\, n\in \mathbb Z_+ \}$,
where $a^{(n)}$ stands for $d^n(a)$,
then $U_d(V) $ is the free associative commutative
algebra generated by $d^{\omega } X$ relative to the 
defining relations obtained by all formal derivatives of
\begin{equation*}
a \oo b = a d(b), \quad a, b \in X, 
\end{equation*}
where $a \oo b $ is treated as a linear combination of elements from $X$.
Then we define a bracket:
\begin{equation*}
\{a^{(n)}, b^{(m)}\} = (m-1) a^{(n+1)} b^{(m)} - (n-1) a^{(n)} b^{(m+1)}, \quad
a, b \in X, \quad n,m \ge 0.
\end{equation*}

This bracket proves to be well-defined, compatible with the Novikov product, and
represents the bracket from $V^{(-)}$.
Hence we can take $U_d(V)$ as as a differential Poisson envelope 
of $V^{(-)}$.
\end{example}

\section{Conformal algebras and conformal endomorphisms}

\subsection{Lie conformal algebras and enveloping}

In this section we recall the main definitions concerning Lie and associative conformal
algebras and their relations to GD algebras. We will not dig deep into origins and
structure results, see for specifics on concerning topics \cite{BDK}, \cite{BKL2003}, \cite{DK}---\cite{FKR}, \cite{Kol2011}---\cite{Kol2016} and \cite{Ro1999}---\cite{Ro2000}. 

Hereinafter let $H$ = $\Bbbk [\partial]$ be a ring of polynomials in one
indeterminate and the field $\Bbbk $ be of characteristic zero.

\begin{definition}
An $H$-module $C$ endowed with a linear operation $[\cdot \ooo{\lambda} \cdot]$:
$C \otimes C \to C[\lambda]$ is called \textbf{Lie conformal algebra} if the following
list of conditions holds:
\begin{equation*}
    [\partial a \ooo{\lambda} b] = - \lambda [a \ooo{\lambda} b], \quad
    [a \ooo{\lambda} \partial b] = (\partial + \lambda) [a \ooo{\lambda} b], \quad
    \text{(sesqui-linearity)}
\end{equation*}
\begin{equation*}
    [a \ooo{\lambda} b] = - [b \ooo{- \partial - \lambda} a], \quad 
    \text{(skew-symmetry)}
\end{equation*}
\begin{equation*}
    [[a \ooo{\lambda} b] \ooo{\lambda + \mu} c] = [a \ooo{\lambda} [b \ooo{\mu} c]] - 
    [b \ooo{\mu} [a \ooo{\lambda}  c], \quad \text{(Jacobi identity)}
\end{equation*}
for every $a, b, c \in C$.
Here we assume that 
$$
a \ooo{\lambda} b = \sum_{n \ge 0} \frac{\lambda^n}{n!} (a \ooo{n} b), \quad
\forall a, b \in C.
$$
The coefficients $(a \ooo{n} b)$ are called the $n$-products of $a $ and $b$.
\end{definition}

In this paper we are particularly interested in a quite wide class of conformal
algebras called \textbf{quadratic conformal algebras}.

\begin{definition}
A Lie conformal algebra $C$ is called a quadratic one if it is free $H$-module, i.e.
$C = H \otimes V$, and its underlying vector space $V$ is 
equipped with three bilinear operations $c_i(a, b)$ : $V \times V \to V$, $i = 0, 1, 2$,
such that
$$
[a \ooo{\lambda} b] = c_0(a,b) + \partial c_1(a,b) + \lambda c_2(a,b),
$$
for every $a, b \in V$.
\end{definition}

It was shown in \cite{Xu2000} that the skew-symmetry and Jacobi identity are equivalent to
the presentation 
$$ c_0(a,b) = [a,b], \quad c_1(a,b)=b\circ a, \quad c_2(a,b)=a\circ b+b\circ a, $$
where $(V;~(\cdot \oo \cdot),~[\cdot, \cdot])$ is a GD algebra. Denote the respective
quadratic conformal algebra as $L(V)$.

The most natural examples, and the most interesting relative to the cause of this paper,
are the following:

\begin{example}
Let $(C;~[\cdot,\cdot])$  be an ordinary Lie algebra. Consider the free $H$-module 
$\Cur C = H\otimes C$ with the bracket
\[
[(f(\partial)\otimes a)\ooo{\lambda} (g(\partial)\otimes b)] = f(-\lambda )g(\partial+\lambda )\otimes [a, b]
\]
Clearly, it is a Lie conformal algebra called {\bf current conformal algebra}. 
This is a particular case of a quadratic conformal algebra 
corresponding to the GD algebra obtained from the given Lie algebra
with trivial Novikov structure.
\end{example}

\begin{example}
The most well-known example of a conformal algebra not happens to be a current one is
\textbf{the Virasoro conformal algebra}. Explicitly,
it is a free $H$-module with one generator $v$ 
relative to the following product:
$$
[v \ooo{\lambda} v] = (\partial + 2 \lambda) v.
$$
This example arises from the 1-dimensional GD algebra $V$ = $\Bbbk v$ with $v \oo v=v$.
\end{example}

Another class of conformal algebras that will be of use in this paper is the class of associative
ones. For their structure theory see \cite{BKL2003}, \cite{Kol2006}, \cite{Kol2007}, \cite{KK2019}, etc.

\begin{definition}
An $H$-module $A$ endowed with a linear operation $(\cdot \ooo{\lambda} \cdot)$: 
$A \otimes A \to A[\lambda]$ is called \textbf{associative conformal algebra} if 
the following list of conditions holds:
\begin{equation*}
    (\partial a \ooo{\lambda} b) = - \lambda (a \ooo{\lambda} b), \quad
    (a \ooo{\lambda} \partial b) = (\partial + \lambda) (a \ooo{\lambda} b), \quad
    \text{(sesqui-linearity)}
\end{equation*}
\begin{equation*}
    ((a \ooo{\lambda} b) \ooo{\lambda + \mu} c) = (a \ooo{\lambda} (b \ooo{\mu} c)),
    \quad \text{(associativity)}
\end{equation*}
for every $a, b, c \in A$.
\end{definition}

Given a pair of elements $a,b\in A$, denote by $N(a,b)=\deg_\lambda (a\ooo\lambda b)+1$ with an exception in the case $(a\ooo\lambda b)=0$:
for the latter, set $N(a,b)=0$. In terms of $n$-products, $N(a,b)$
is the minimal $N$ such that $(a\ooo{n} b)=0$ for all $n\ge N$.
The function $N(x,y)$ is called locality function on $A$.

In the case of "ordinary" algebras there exists a well-known way how to turn an
associative algebra into a Lie algebra. Quite similar construction appears in the 
conformal case as well. 
Let $A$ be an associative conformal algebra with product $(\cdot \ooo{\lambda} \cdot)$.
The underlying space endowed with the bracket as follows,
$$
[a \ooo{\lambda} b] = (a \ooo{\lambda} b) - (b \ooo{-\partial-\lambda} a), \quad 
\forall a, b \in A,
$$
is a Lie conformal algebra. Denote it, after the "ordinary"
case, as $A^{(-)}$.

Now, a natural question comes in mind: whether the inverted construction is the case 
as well.

\begin{definition}
Given a Lie conformal algebra $C$ and an associative conformal
algebra $A$, we say $A$ is an associative enveloping conformal  algebra 
of $C$ if there exists a
homomorphism $\tau$: $C \to A^{(-)}$ of conformal algebras such that $A$ is
generated (as an associative conformal algebra) by the image of $C$.
\end{definition}
Denote such a pair as $(A,~\tau)$.

The most interesting of associative envelopes is the universal envelope --- the universal
object in the class, quite resembling the ordinary case. However, it appears the locality
function to be the fine point: there is no sense in speaking about the utterly universal
object. 


A proper definition of the ``conditional'' universal enveloping  conformal
algebra was given in
\cite{Ro2000}:
\begin{definition}
Given a Lie conformal algebra $C$ generated as an $H$-module by its subset $B$, and a
function $N$: $B \times B \to Z_+$. An associative conformal enveloping algebra 
$(\mathcal{U}(C;~N),~\tau_U)$ is called the \textbf{universal one relative to the locality
function $N$} if the following universal property holds: for every associative 
envelope $(A,~\tau)$ such that $N(\tau(a), \tau(b)) \le N$ for all $a, b \in B$, there
exists a complementary homomorphism of associative conformal algebras $\varphi$:
$\mathcal{U}(C;~N) \to A$, such that $\tau = \varphi \circ \tau_U$: $C \to A^{(-)}$.
\end{definition}
The universal conformal enveloping algebra $(\mathcal{U}(C;~N),~\tau_U)$ does exist for any
Lie conformal algebra $C$ and locality function $N$, for the explicit construction see
\cite{Ro2000}. Unfortunately, $\tau_U$ will not necessarily be injective. Even more is 
true, it was shown ibid that there
exist Lie conformal algebras $C$ such that $\tau$ is not injective for every 
$(A, \tau)$, i.e., $C$ can not be embedded into an associative conformal algebra at all.
For another approach to construction of associative conformal universal envelopes
see \cite{Kol2020IJAC} (also \cite{KK2022}).

\subsection{Conformal endomorphisms} 
Let us recall the definition of a conformal 
endomorphism and some related constructions.

\begin{definition}
Let $M$ be a unital left $H$-module. A \textbf{conformal endomorphism} of $M$ is such an
element $a$ that there is a linear map $a_{\lambda}$: $M \to M[\lambda]$ and
$a_{\lambda} \partial = (\partial + \lambda) a_{\lambda}$.
\end{definition}

Denote the set of conformal endomorphisms of $M$ by $\Cend{M}$.
Note that $\Cend{M}$ may be equipped with the structure of an $H$-module by the rule:
$$
(\partial a)_{\lambda} = - \lambda a_{\lambda}, \quad a \in \Cend{M}. 
$$

Moreover, if $M$ is a finitely generated $H$-module then $\Cend{M}$ is an associative
conformal algebra itself with the following multiplication rule: 
\begin{equation} \label{mult}
(a \ooo{\lambda} b)_{\mu} = a_{\lambda} b_{\mu - \lambda}, \quad a, b \in
\Cend{M}.
\end{equation}

Unfortunately, be $M$ infinitely generated, the presented multiplication rule fails due 
to locality. Though one can make out a subspace of finitary
conformal endomorphisms 
$\Cendf{M} \subseteq \Cend{M}$ that still proves to be an associative conformal algebra.
The idea is to consider only those conformal endomorphisms that have 
bounded on powers of $\lambda$ for the 
images on generators of $M$. Explicitly, let $B = \{b_1, b_2, b_3, \dots \}$ be 
a set of generators of $M$ as of an $H$-module,
$$
\Cendf{M} = \{a \in \Cend{M} \mid \exists n(a) \in \mathbb{N}: \text{deg}(a_{\lambda}b_i) 
\le n(a) \}.
$$
Clearly, it is closed under the multiplication rule \eqref{mult}.
Note that the definition of $\Cendf{M}$ depends on the choice of 
a system of generators $B$ in the case when $M$ is infinitely generated.

The inclusion is actually equality if $M$ is finitely generated. 
In general, though, this is not the case. 
Indeed, if one define $a_{\lambda}b_i = b_i \lambda^i$ for every $b_i \in B$, then
$a \notin \Cendf{M}$. The  $\lambda$-product of 
two suchlike conformal endomorphisms is not a polynomial on $\lambda$
but rather an infinite series. For example,
$$
(a \ooo{\lambda} a)_{\mu} b_i = a_{\lambda}(a_{\mu - \lambda} b_i) = 
b_i \lambda^i (\mu - \lambda)^i.
$$
Plainly, for any $n \in \mathbb{N}$ there is $i \in \mathbb{N}$ such that 
$a \ooo{n} a \ne 0$, a contradiction with locality.

Explicitly, the space $\Cendf{M}$ for free $H$-module $M = H \otimes V$ can be
uniquely presented as follows
(see \cite{Ret2001}), up to isomorphism:
$$
\Cendf{M} \cong H \otimes (\End{V})[x] \cong \Bbbk [\partial, x] \otimes \End{V},
$$
where $\End{V}$ is the space of "usual" endomorphisms of the vector space $V$.
Conformal action is as follows,
\begin{equation}\label{action}
(f(\partial, x) \otimes \alpha)_{\mu} (h(\partial) \otimes u) = 
f(-\mu, \partial)h(\partial + \mu) \otimes \alpha(u),
\end{equation}
where $f(\partial, x) \in \Bbbk [\partial, x]$, $h(\partial) \otimes u \in M$ and
$\alpha \in \End{V}$.

In details, the associative conformal multiplication \eqref{mult} is given by the rule,
\begin{equation}\label{product}
    (f(\partial, x) \otimes \alpha) \ooo{\lambda} (g(\partial, x) \otimes \beta) = 
    f(-\lambda, x)g(\partial+\lambda, x + \lambda) \otimes \alpha \beta,
\end{equation}
where $f(\partial, x), g(\partial, x) \in \Bbbk [\partial, x]$ and 
$\alpha, \beta \in \End{V}$.

\section{Universal associative conformal envelope of locality N = 3}

\subsection{Existence}

In this section we present the approach that allows us to establish that every quadratic
conformal algebra constructed on a special GD algebra embeds into its universal enveloping associative
conformal  algebra with locality bound $N = 3$.

Let $V$ be a special GD algebra and let $P$ be its differential Poisson enveloping. Say,
without loss of generality, that $P$ has an identity element on the associative commutative
multiplication (if it does not have one then we consider $P^{\#}=P\oplus \Bbbk 1$). 

Recall that $\ad_{a}$: $x\mapsto \{a,x\}$ stands for the adjoint operator,
$L_{a}$: $x\mapsto ax$ for the operator of commutative multiplication, 
and $\id = \id_{P}$ for the identity operator on $P$. All 
these maps are extended $H$-linearly to linear operators on
$H\otimes P$. 

Hereinafter, in order to avoid confusion 
between $d$ and $\partial$, replace the latter with $T$. Also, let us sometimes omit the
$\otimes$-sign to make the computations less cumbersome.

Let us state the following helpful lemma

\begin{lemma}\label{lemma}
Let $V$ be a special GD algebra and $P(V)$ be its differential Poisson
enveloping. Let $L(V)$ be the quadratic conformal algebra related to $V$. Then the 
$H$-linear map 
$$
\tau: L(V) \to \Cendf{(H \otimes P(V))}^{(-)}, 
$$
defined on $H$-generators as follows:
\begin{equation} \label{hom}
    \tau(a) = \ad_{a} + x [d, L_{a}] - T (d + \id) L_{a}, \quad \forall a\in V,
\end{equation}
is an injective homomorphism of conformal Lie algebras.
\end{lemma}

\begin{proof}
Recall that $a^{(n)}$ stands for $d^n(a)$, for $n = 1, 2$ assign simply $a'$, $a''$.
First, let us compute how the
considered operators commute in $\End P$, modulo the Leibniz rule:
\begin{align}
L_{b} \ad_{a} c &= b \cdot \{a, c\} = \{a, b \cdot c\} - \{a, b\} \cdot c = 
(\ad_{a} L_{b} - L_{\{a, b\}}) c, \label{Lad}
\\
d \ad_{a} c &= d \{a, c\} = \{a', c\} + \{a, c'\} = (\ad_{a'} + \ad_{a} d) c, \label{dad}
\\
d L_{b} c &= d (b \cdot c) = b' \cdot c + b \cdot c' = (L_{b'} + L_{b} d) c,
\label{dL}
\end{align}
for every $a, b, c \in P$. Then \eqref{hom} turns into $\tau(a)=\ad_{a} + 
x L_{a'} - T (L_{a'} + L_{a}(d + \id))$ for every $a \in V$.

Now, compute $\tau(a)\ooo{\lambda}\tau(b)$:
\begin{multline} \label{alambdab}
(\ad_{a} + x L_{a'} - T (L_{a'} + L_{a}(d + \id)))
\ooo{\lambda} (\ad_{b} + x L_{b'} - T (L_{b'} + L_{b}(d + \id))) 
\\
= (\ad_{a} + x L_{a'} + \lambda (L_{a'} + L_{a}(d + \id))) 
(\ad_{b} + (x +\lambda) L_{b'} - (T + \lambda) (L_{b'} + L_{b}(d + \id)))
\\
= (\ad_{a} + (x + \lambda) L_{a'} + \lambda L_{a}(d + \id)) 
(\ad_{b} + (x - T) L_{b'} - (T + \lambda) L_{b}(d + \id)).
\end{multline}

Open the parentheses and gather terms at
the powers of $\lambda$. Then order operators
by the rule $\ad < L < d$, according to \eqref{Lad} - \eqref{dL} and 
$L_a L_b = L_{a\cdot b}$. 
The coefficient at $\lambda^0$ is
\begin{multline}\label{lam0}
\ad_{a}\ad_{b} + (x - T) \ad_{a} L_{b'} - T \ad_{a} L_{b}(d + \id)
\\
+ x L_{a'}\ad_{b} + x(x - T) L_{a'} L_{b'} - x T L_{a'} L_{b}(d + \id) 
\\
= \ad_{a}\ad_{b} + (x - T) \ad_{a} L_{b'} - T \ad_{a} L_{b}(d + \id) 
\\
+ x (\ad_{b}L_{a'} - L_{\{b, a'\}}) + x(x - T) L_{a'\cdot b'} 
- xT L_{a'\cdot b} (d + \id).
\end{multline}

As a coefficient at $\lambda$, we have
\begin{multline*}
- \ad_{a} L_{b}(d + \id) - x L_{a'} L_{b}(d + \id) + L_{a'} \ad_{b} +
(x - T) L_{a'} L_{b'} - T L_{a'} L_{b}(d + \id) 
\\
+ L_{a}(d + \id) \ad_{b} 
+ (x - T) L_{a}(d + \id) L_{b'} - T L_{a}(d + \id) L_{b}(d + \id), 
\end{multline*}
which, after some routine, becomes
\begin{multline} \label{lam1}
- \ad_{a} L_{b}(d + \id) - (x + T) L_{a'\cdot b}(d + \id) + 
\ad_{b}(L_{a'} + L_a (d+id)) - L_{\{b, a\}'} + \ad_{b'} L_a 
\\
- L_{\{b, a\}} (d + \id) + (x - T) (L_{a\cdot b''} + L_{a'\cdot b'}) 
+ (x - 2T) L_{a\cdot b'} (d + \id) - T L_{a\cdot b} (d + \id)^2. 
\end{multline}

The coefficient at $\lambda^2$ is
\begin{multline} \label{lam2}
- L_{a'} L_b (d + \id) - L_a (d + \id) L_b (d + \id) = - L_{(a\cdot b)'} (d + \id) 
- L_{a\cdot b} (d + \id)^2.
\end{multline}

In order to compute $\tau(b)\ooo{- T - \lambda} \tau(a)$ one need to switch $a$ and $b$,
substitute $\lambda$ by $-T-\lambda$, and gather similar terms on powers of $\lambda$
in the original formula \eqref{alambdab}.
Then, at $\lambda^0$ we have
\begin{multline*}
\ad_{b}\ad_{a} + (x - T) \ad_{b} L_{a'} - T \ad_{b} L_{a}(d + \id) 
+ x (\ad_{a}L_{b'} - L_{\{a, b'\}}) 
+ x(x - T) L_{b'\cdot a'} - xT L_{b'\cdot a} (d + \id), 
\end{multline*}
and the addition from the other groups
\begin{multline*}
- T (- \ad_{b} L_{a}(d + \id) - (x + T) L_{b'\cdot a}(d + \id) 
+ \ad_{a}(L_{b'} + L_b (d+id)) - L_{\{a, b\}'} 
\\
+ \ad_{a'} L_b - L_{\{a, b\}} (d + \id) 
+ (x - T) (L_{b\cdot a''} + L_{b'\cdot a'})
+ (x - 2T) L_{b\cdot a'} (d + \id) 
\\
- T L_{b\cdot a} (d + \id)^2)
+ T^2 (- L_{(b\cdot a)'} (d + \id) - L_{b\cdot a} (d + \id)^2).
\end{multline*}
Gather similar terms:
\begin{multline} \label{lam00}
\ad_{b}\ad_{a} + (x - T) \ad_{b} L_{a'} + (x - T) \ad_{a}L_{b'} - x L_{\{a, b'\}}
\\
+ (x - T)^2 L_{b'\cdot a'} - T \ad_{a} L_b (d+id) 
+ T L_{\{a, b\}'} - T  \ad_{a'} L_b + T L_{\{a, b\}} (d + \id) 
\\
- T (x - T) L_{b\cdot a''} - T (x - T) L_{b\cdot a'} (d + \id).
\end{multline}

The coefficient at $\lambda$ is
\begin{multline*}
\ad_{b} L_{a}(d + \id) + (x + T) L_{b'\cdot a}(d + \id) - 
\ad_{a}(L_{b'} + L_b (d+id)) + L_{\{a, b\}'} - \ad_{a'} L_b 
\\
+ L_{\{a, b\}} (d + \id) - (x - T) (L_{b\cdot a''} + L_{b'\cdot a'}) 
- (x - 2T) L_{b\cdot a'} (d + \id) + T L_{b\cdot a} (d + \id)^2,
\end{multline*}
and the addition
$$
2T (- L_{(b\cdot a)'} (d + \id) - L_{b\cdot a} (d + \id)^2).
$$
Gather similar terms:
\begin{multline}\label{lam11}
\ad_{b} L_{a}(d + \id) + (x - T) L_{b'\cdot a}(d + \id) - 
\ad_{a}(L_{b'} + L_b (d+id)) + L_{\{a, b\}'} - \ad_{a'} L_b 
\\
+ L_{\{a, b\}} (d + \id) - (x - T) (L_{b\cdot a''} + L_{b'\cdot a'}) 
- x L_{b\cdot a'} (d + \id) - T L_{b\cdot a} (d + \id)^2.
\end{multline}

The expression at $\lambda^2$ is exactly \eqref{lam2} due to the product in $P$ is
commutative.

From the other side,  
\begin{multline} \label{t(ab)}
\tau([a \ooo{\lambda} b]) = \tau([a, b] + T (b\oo a) + \lambda (a\oo b + b\oo a)) 
\\
= \ad_{\{a, b\}} + (x - T) L_{\{a, b\}'} - T L_{\{a, b\}}(d + \id) 
\\
+ T (\ad_{a'\cdot b} + (x - T) L_{(a'\cdot b)'} - T L_{a'\cdot b} (d + \id)) 
\\
+ \lambda (\ad_{(a\cdot b)'} + (x - T) L_{(a\cdot b)''} - T L_{(a\cdot b)'} (d + \id))
\end{multline}

Now, consider the operator 
\[
\tau(a) \ooo{\lambda} \tau(b) - \tau(b) \ooo{- T- \lambda} 
\tau (a) - \tau([a \ooo{\lambda} b]) \in \Cendf{(H \otimes P(V))}
\]
and verify that it is
identical to the zero operator. 
Compare the coefficients at
$\lambda$ found in \eqref{lam1}, \eqref{lam11}, \eqref{t(ab)}: 
the respective sum, modulo some routine transformations, equals to
$$
1 \otimes (\ad_{b} L_{a'} + \ad_{b'} L_a + \ad_{a} L_{b'} + \ad_{a'} L_b 
- \ad_{(a\cdot b)'}), \quad  a, b \in V.
$$
This is essentially a linear operator acting on $P$, so, for any $c \in P$:
\begin{multline*}
1 \otimes (\ad_{b} L_{a'} + \ad_{b'} L_a + \ad_{a} L_{b'} + \ad_{a'} L_b 
- \ad_{(a\cdot b)'})_{\mu}c
\\
= \{b, a'\cdot c\} + \{b', a\cdot c\} + \{a, b'\cdot c\} + \{a', b\cdot c\}
- \{a'\cdot b, c\} - \{a\cdot b', c\}.
\end{multline*}
Apply the Leibniz rule to note that the RHS of the last relation vanishes, as demanded.

The coefficient at $\lambda^0$ obtained from \eqref{lam0}, \eqref{lam00}, 
\eqref{t(ab)}, after cancellation of similar terms,
turns into
\begin{equation}
\ad_{a}\ad_{b} - \ad_{b}\ad_{a} - \ad_{\{a, b\}} 
+ T  \ad_{a'} L_b + T \ad_{b} L_{a'} - T \ad_{a'\cdot b}.
\end{equation}
Here the first three terms are exactly Jacobi identity and the last three terms represent the Leibniz rule. Hereby, $\tau$
is indeed a homomorphism.

As for injectivity, 
assume
$v \in L(V)$ and $\tau(v)=0$. 
Let $v = \sum_{s \ge 0} T^s \otimes a_s$, $a_s \in V$.
Then by \eqref{action} we have
%
\begin{multline*}
\tau(v)_{\mu} 1 = \sum_{s \ge 0} T^s (\ad_{a_s} + (x - T) L_{a_s'} 
- T L_{a_s} (d + \id))_{\mu}1 
= \sum_{s \ge 0} (- \mu)^s ((T + \mu) L_{a_s'} + \mu L_{a_s} )1 
\\
= \sum_{s \ge 0}(-\mu)^s T L_{a_{s}'}1 - \sum_{s \ge 0} 
(-\mu)^{s+1} (L_{a_s'} + L_{a_{s}})1 
=
\sum_{s \ge 0}(-\mu)^s T a_{s}' - \sum_{s \ge 0} 
(-\mu)^{s+1} (a_s' + a_{s})  
\end{multline*}
If all coefficients at $\mu^{s}$, $s\ge 0$, are zero 
then we obtain the recurrent conditions $a_s' = 0$, $a_s' + a_s = 0$, $s\ge 0$.
Hence, all $a_s$ are zero and $\tau(v)=0$ implies $v=0$, as demanded.
\end{proof}

Now, we are able to prove the main statement of the presented work.

\begin{theorem}
Let  $L(V)$ be a quadratic conformal algebra related to a special Gel'fand--Dorfman algebra $V$. 
Then $L(V)$ embeds into its universal associative enveloping conformal
algebra $\mathcal{U}(L(V);~N)$ with respect to the associative locality bound $N = 3$.
\end{theorem}

\begin{proof}
Due to the previous reasoning along with Lemma 1, there exists
an injective homomorphism
of conformal Lie algebras
$\tau$: $L(V) \to \Cendf{(H \otimes P(V))}^{(-)}$, 
where $P(V)$ is a unital differential
Poisson enveloping of $V$. Consider the associative conformal subalgebra in $\Cendf{(H\otimes P(V))}$ 
generated by $\tau(V)$, refer
to it as $A(V)$.

The introduced $A(V)$ is an associative envelope
 of $L(V)$ for the associative
locality bound $N = 3$. The latter 
follows from the computation 
of $\tau(a)\ooo\lambda \tau(b)$ in the proof of Lemma 1:
the degree in $\lambda $ of this polynomial 
does not exceed two. 
Note that the 2nd product of $\tau(a)$ and $\tau(b)$ may not be zero (see \eqref{lam2}).

The universal property grants that there exists 
a homomorphism of associative conformal algebras 
$\varphi : \mathcal{U}(L(V);~N)\to A(V)$ 
extending the  canonical map $L(V)\to \mathcal{U}(L(V);~N)$
for $N = 3$.
Hence, $L(V)$ embeds into the respective 
universal enveloping associative 
conformal algebra. 
\end{proof}

\begin{remark}
The theorem proves to be true in the case of superalgebras as well, one just need slightly
adjust all the definitions and calculations, the reasoning above will do perfectly.
\end{remark}

\subsection{Isomorphism}
The surjective homomorphism $\varphi:
\mathcal{U}(L(V);~N) \to A(V)$ mentioned above 
is not injective in general, but sometimes it is.
This is interesting to note that for the case 
when $L(V)$ is the Virasoro 
conformal algebra the associative envelope $A(V)$
constructed above is the universal one (corresponding to 
the locality bound $N=3$).

\begin{example}
Consider the GD algebra $V = \Bbbk \< v \>$, 
where $v \oo v = v$.
Then $L(V) = \mathrm{Vir} = Hv$ 
with the following $\lambda $-product on the generator 
$$
[v\ooo{\lambda} v] = (T + 2\lambda) v.
$$
Then the single conformal homomorphism 
$\tau(v) = x \otimes \id - T \otimes (d + \id) L_v$ 
generates an associative conformal subalgebra
$A(V)$ in $\Cendf{(H \otimes P(V))}$, where $P(V) = \Bbbk [v]$ and 
$d = \frac{d}{dv}$. 
It is more convenient to treat $P(V)$ as a
module generated by $v$ over the (associative) Weyl algebra 
$W = \Bbbk \<d, p| dp - pd = 1\>$. Here
$p$ stands for $L_v$ and $1$ for $\id_{P(V)}$.
Then $\varphi(v) = x - T(d+1)p$.

In \cite{Kol2020} it was shown that the $H$-linear basis of 
$\mathcal{U}(Vir;~3)$ 
is
the following one:
\begin{align}
(L_0)^s v&, \quad s \ge 0, \label{L0}
\\
(L_0)^q (L_1)^l L_2 v&, \quad q, l \ge 0. \label{L012}
\end{align}
Here $L_n$, $n\ge 0$, stands for the linear operator $(v \ooo{n} \cdot)$ on 
$\mathcal{U}(Vir;~3)$.

Now, let us calculate the images of \eqref{L0}, \eqref{L012} in $A(V)$.
Start with 
\begin{multline*}
\varphi(v) \ooo{\lambda} \varphi(v) = (x - T (d + 1)p) \ooo{\lambda} 
(x - T (d + 1)p) 
\\
= x(x - T (d+1)p) + \lambda (x - T ((d+1)p)^2)
+ \lambda^2((d+1)p (1 - (d+1)p)).
\end{multline*}
Hence, $\varphi(v \ooo{0} v) = x(x - T (d+1)p)$ and it is the matter of easy induction 
to deduce that 
\begin{equation} \label{imL0}
    \varphi((L_0)^s v) = x^s(x - T (d+1)p), \quad s \ge 0.
\end{equation}

We obtain as well that $ \varphi(v \ooo{2} v) = 2(d+1)p (1 - (d+1)p)$. 
Note that
$\varphi(v \ooo{1} v)$ is linearly dependent with $\varphi(v)$ and
$\varphi(v \ooo{2} v)$ over $H$, that is an exact consequence of 
the defining relations in $\mathcal{U}(Vir;~3)$.

Now, $2 \varphi(v \ooo{\lambda} (v \ooo{2} v))$ is as follows:
$$
(x + \lambda (d+1)p)(d+1)p(1 - (d+1)p) = x (d+1)p(1 - (d+1)p) + 
\lambda ((d+1)p)^2(1 - (d+1)p).
$$
Hence, the application of $L_1$ to $\varphi(v \ooo{2} v)$ is just
the multiplication by $(d+1)p$ from the left, and the action of 
$L_0$ is just multiplication by $x$.
 Thus,
\begin{equation} \label{imL012}
\varphi((L_0)^q (L_1)^l L_2 v) = \frac{1}{2} x^q ((d+1)p)^l (1 - (d+1)p), \quad q, l 
\ge 0.
\end{equation}

Suppose there is some nontrivial linear combination 
$$
\sum_{s \ge 0}h_s(T)
x^s(x - T (d+1)p) + \sum_{q,l \ge 0}h'_{q,l}(T) x^q ((d+1)p)^l (1 - (d+1)p) = 0,
$$
where $h_s(T), h'_{q,l}(T) \in H$. Apply the expression in the LHS to 1,
modulo \eqref{imL0}, \eqref{imL012}:
\begin{multline} \label{lincomb}
0 = \sum_{s \ge 0}h_s(-\mu) T^s(T + \mu (d+1)p)1
+ \sum_{q,l \ge 0}h'_{q,l}(-\mu) T^q ((d+1)p)^l (1 - (d+1)p)1
\\
= \sum_{s \ge 0}h_s(-\mu) (T^{s+1} + \mu T^s)1 + 
\sum_{s \ge 0}h_s(-\mu) \mu T^s v
- \sum_{q,l \ge 0}h'_{q,l}(-\mu) T^q \sum_{i = 1}^{l+1} v^i c_i,
\end{multline}
where the coefficients $c_i$ are the following:
$$
c_{l+1} = 1, \quad 
c_i = \sum_{2 \le j_1 \le \dots \le j_{l+1-i} \le 1+i} \prod_{k = 1}^{l+1-i} j_i,
\quad i = \overline{1,l}.
$$
Consider $s = s_{\max}$, hence, the first sum in the RHS of \eqref{lincomb} gives
$h_{s_{\max}} \equiv 0$. Then gradually descend up to $s = 0$, thus 
$h_s \equiv 0$. Hereby, the first two sums in \eqref{lincomb} vanish. 
In the same fashion, regard the last sum for $l = l_{\max}$ and
obtain $h'_{q,l_{\max}} \equiv 0$, for all $q$, from the coefficient 
at $v^{1+l_{\max}}$. Similarly, descend up to $l = 0$, thus $h'_{q,l} \equiv 0$.
A contradiction.

Therefore, we establish that, given a Lie conformal algebra $L(V) = Vir$ related 
to the GD algebra $V = \Bbbk \<v\>$ with the multiplication rule $v \oo v = v$,
there is the isomorphism of associative conformal algebras
$$
\Cendf{(H \otimes P(V))} \supset A(V) \cong \mathcal{U}(Vir;~3).
$$
\end{example}

\begin{remark}
Mention should be made of yet another equivalent presentation for 
$\mathcal{U}(Vir;~3)$, that is given rise for by the adjoint representation, namely,
$$
v \mapsto [v \ooo{\lambda} \cdot] = x \otimes \id - T \otimes dp \in 
\Cendf{(H \otimes P(V))}. 
$$
The $T$-linear basis itself is
$$
x^q \otimes (dp)^l (1 - dp) \quad \text{and} 
\quad x^{s+1} \otimes \id - x^s T \otimes dp, \quad q, l, s \ge 0,
$$
The computations are quite resembling the ones just made.
\end{remark}

Such an isomorphism is not always the case. An example as simple as the
following one will do.

\begin{example}
Take the GD algebra $V = \Bbbk \< v \>$, where $v \oo v = 0$. Consider 
$L(V) = \Cur{V} \cong H \otimes V$, it is an abelian Lie conformal algebra.
If we consider 
 $P(V) = \Bbbk [v]$ as above, and $d \equiv 0$, 
 then $\tau(v) = 1 \otimes \ad_v - T \otimes L_v$.
However, the associative conformal subalgebra
$A(V)$ generated in $\Cendf{(H \otimes P(V))}$ by $\tau (v)$
is not isomorphic to the respective universal envelope.
For an abelian Lie algebra, the universal enveloping 
associative conformal algebra is the free commutative one.
It has the same basis as mentioned in \eqref{L0}, \eqref{L012},
but, for example,
$$
\varphi(L_0 \dots L_0 L_1 \dots L_1 L_2 v) 
= - 2 \ad_v\dots \ad_v L_v\dots L_v = 0.
$$ 


What is more interesting is that 
even if we choose the richest Poisson differential enveloping algebra 
for $V$ (the universal one) then the kernel of $\varphi $ is still nontrivial.

Denote the universal Poisson differential enveloping algebra 
of a GD algebra $V$ as $P_d(V)$.
In the case when $V$ is 1-dimensional algebra with zero operations,
we have 
$$
P_d(V) = \faktor{(\Pois\< v, v', v'', \dots \>^{\#})}{I_V},
$$
where $I_V $ is the differential ideal in the free Poisson algebra
$\Pois\< v, v', v'', \dots \>$
generated by $v\cdot v'$. 
Let us prove the helpful lemma.

\begin{lemma}\label{lem:PoisRel}
In the Poisson algebra $P_d(V)$ as above
the following equalities hold:
\begin{equation}\label{conseq}
0 = \{v,~\{v^{(k_1)},~\{ \dots \{v^{(k_{m-1})},~v^{(k_m)}\}\dots \} \cdot v,
\quad \text{for all}~k_1,~\dots,~k_m \ge 0.
\end{equation}
\end{lemma}
\begin{proof}
By induction on $m$.
For $m = 0$, let us compute some compositions in $P_d(V)$
following the ideas of \cite{BokCZ17}. 
First, since $I_V$ is a 
differential ideal, we have  
$\{v^{(m)}, (v\cdot v')^{(n)} \}=0$, for all $m,n \ge 0$.
Apply the Leibniz rule to obtain 
\begin{equation} \label{comp}
0 = \sum_{i \ge 0}\binom{n}{i} \{v^{(m)},~v^{(i)} \cdot v^{(n + 1 - i)}\} = 
\sum_{i \ge 0}\binom{n+1}{i} \{v^{(m)},~v^{(i)}\} \cdot v^{(n + 1 - i)}, \quad
n,~m \ge 0.
\end{equation}
Choose some $k \ge 0$,
multiply \eqref{comp} by $\binom{k+1}{m}$
 and make summation over $m = \overline{0,k}$ with 
 $n = k - m $:
\begin{equation*}
0 = \sum_{m = 0}^{k} \binom{k+1}{m} \sum_{i = 0}^{k-m+1} \binom{k-m+1}{i}
\{v^{(m)},~v^{(i)}\} \cdot v^{(k + 1 - i - m)}.
\end{equation*}
Let us change the order of summation. 
Consider the diagonal $i + m = l$, 
$l = \overline{0,k+1}$. 
Then the last equation turns into
\begin{equation} \label{comp1}
0 = \sum_{l = 0}^{k+1} \sum_{\substack{i, m = 0 \\ i+m = l}}^{k} \binom{k+1}{m}
\binom{k-m+1}{i} \{v^{(m)},~v^{(i)}\} \cdot v^{(k + 1 - l)} 
+ \{v,~v^{(k+1)}\}\cdot v,
\end{equation}
where the product of the binomial coefficients, 
$$
\binom{k+1}{m}\binom{k-m+1}{i}
= \frac{(k+1)!(k+1-m)!}{m!(k+1-m)!i!(k+1-m-i)!} =
\frac{(k+1)!}{m!i!(k+1-m-i)!},
$$ 
is symmetric on the indices $i$ and $m$. 
Hence, having fixed some $l$, the respective
layer in the sum in \eqref{comp1} vanishes due to skew-symmetry of the 
bracket. Therefore, we obtain $\{v,~v^{(k+1)} \}\cdot v=0$ as demanded.

Now, let the statement be true for $m - 1$. Denote 
$\{v^{(k_1)},~\{ \dots \{v^{(k_{m-2})},~v^{(k_{m-1})}\}\dots \} = u$, hence,
$\{v,~u\}\cdot v = 0$ by assumption. Consider $\{v^{(p)},~\{v,~u\}\cdot v\}=0$,
for $p \ge 0$, and apply the Leibniz rule:
\begin{equation*}
0 = \{v^{(p)},~\{v,~u\}\cdot v\} = \{v^{(p)},~\{v,~u\}\} \cdot v + 
\{v,~u\} \cdot \{v^{(p)},~v\}.
\end{equation*}
Modulo the Jacobi identity and skew-symmetry, the equation becomes:
\begin{multline*}
0 = \{\{v^{(p)},~v\},~u\} \cdot v + \{v,~\{v^{(p)},~u\}\} \cdot v + 
\{u,~v\} \cdot \{v,~v^{(p)}\} 
\\
= \{u,~\{v,~v^{(p)}\}\} \cdot v + \{u,~v\} \cdot \{v,~v^{(p)}\} + 
\{v,~\{v^{(p)},~u\}\} \cdot v 
\\
= \{u,~\{v,~v^{(p)}\} \cdot v\} + \{v,~\{v^{(p)},~u\}\} \cdot v,
\end{multline*}
where the first term is already proven to be naught. Hereby,
$$
\{v,~\{v^{(p)},~u\}\} \cdot v = 0, \quad \text{for all }p \ge 0,
$$
and we are done.
\end{proof}

The Leibniz rule immediately implies 

\begin{corollary}\label{cor:final}
For every $f\in P_d(V)$ we have 
$\{v, f\}\cdot v = 0$.
\end{corollary}

The generator $\varphi(v)$ of $A(V)$ is given by the conformal endomorphism
$$
\varphi(v) = \tau(v) = 1 \otimes \ad_v + (x - T) \otimes L_{v'} - 
T \otimes L_v (d + \id)
$$
of the free $H$-module $H\otimes P_d(V)$.
Consider the element
$u=v \ooo{0} (v \ooo{2} v) \in \mathcal{U}(L(V);~3)$
from its $H$-linear basis.
Then $\varphi (u) \in A(V)$ may be calculated as follows,
using \eqref{Lad} -- \eqref{dL}:
evaluate $\tau(v) \ooo{\mu} (\tau(v) \ooo{\lambda} \tau(v))$
and choose the coefficient at $\mu^0
\frac{\lambda^2}{2}$.
As a result, we obtain
\begin{multline*}
\varphi(u) = \tau(v) \ooo{0} (\tau(v) \ooo{2} \tau(v)) 
= (\ad_v + x L_{v'})
(-2 L_v L_v (d + \id)^2) \\
= -2 \ad_v L_v L_v (d + \id)^2 
- 2 x L_{v'} L_v L_v (d + \id)^2,
\end{multline*}
where the second term is zero in $P_d(V)$ due to the multiplication table. 
Recall that action of $A(V)$ on 
$H\otimes P_d(V)$ is faithful and consider, applying the Leibniz rule,
\begin{multline} \label{image}
\tau(v) \ooo{0} (\tau(v) \ooo{2} \tau(v))_{\mu}f = 
-2 \ad_v L_v L_v (d + \id)^2 f \\
= -2 \{v,~v^2 (d + \id)^2f\} = -2 \{v,~f'' + 2 f' + f\} \cdot v^2,
\end{multline}
for every $f \in P_d(V)$. 
By corollary \ref{cor:final},
\eqref{conseq} annihilate the RHS of \eqref{image}. Thus, we are finished.
\end{example}

\end{document}